\documentclass[a4paper,11pt,reqno]{amsart}
\usepackage{amsmath}
\usepackage{amsthm,enumerate}

\usepackage{graphicx}
\usepackage{amssymb}

\usepackage{appendix}

\usepackage[italian,english]{babel}

\usepackage[utf8x]{inputenc}
\usepackage{fancyhdr}

\usepackage{calc}
\usepackage{url}

\usepackage[text={6.3in,8.6in},centering]{geometry}

\usepackage{srcltx, inputenc}

\usepackage[T1]{fontenc}

\usepackage[usenames,dvipsnames]{color}

\makeatletter
\def\cleardoublepage{\clearpage\if@twoside \ifodd\c@page\else%
         \hbox{}%
     \thispagestyle{empty}
     \newpage%
     \if@twocolumn\hbox{}\newpage\fi\fi\fi}
\makeatother

\theoremstyle{plain}
\newtheorem{theorem}{Theorem}[section]

\newtheorem{definition}[theorem]{Definition}

\newtheorem{lemma}[theorem]{Lemma}

\newtheorem{proposition}[theorem]{Proposition}
\newtheorem{remark}[theorem]{Remark}

\numberwithin{equation}{section}
\theoremstyle{definition}

\newcommand{\R}{\ensuremath{\mathbb{R}}}

\begin{document}

\title[]{Blow-up and global existence for \\semilinear parabolic equations \\ on infinite graphs}

\author{Gabriele Grillo}
\address{\hbox{\parbox{5.7in}{\medskip\noindent{Dipartimento di Matematica,\\
Politecnico di Milano,\\
   Piazza Leonardo da Vinci 32, 20133 Milano, Italy.
   \\[3pt]
        \em{E-mail address: }{\tt
          gabriele.grillo@polimi.it
          }}}}}

\author{Giulia Meglioli}
\address{\hbox{\parbox{5.7in}{\medskip\noindent{Fakult\"at f\"ur Mathematik,\\
Universit\"at Bielefeld,\\
   33501, Bielefeld, Germany.
   \\[3pt]
        \em{E-mail address: }{\tt
          gmeglioli@math.uni-bielefeld.de
          }}}}}

\author{Fabio Punzo}
\address{\hbox{\parbox{5.7in}{\medskip\noindent{Dipartimento di Matematica,\\
Politecnico di Milano,\\
   Piazza Leonardo da Vinci 32, 20133 Milano, Italy. \\[3pt]
        \em{E-mail address: }{\tt
          fabio.punzo@polimi.it}}}}}



\maketitle

\maketitle              

\begin{abstract} We investigate existence of global in time solutions versus blow-up ones for the semilinear heat equation posed on infinite graphs. The source term is a general function $f(u)$, and the different behaviour of solutions is characterized by the behaviour of $f$ near the origin and by the first eigenvalue $\lambda_1(G)$ of the negative Laplacian on the graph, which is assumed to satisfy $\lambda_1(G)>0$. In particular, if $f'(0)>\lambda_1(G)$ than all positive nontrivial solution blows up in finite time, whereas if $f'(0)<\lambda_1(G)$, or if a weaker condition involving the Lipschitz constant of $f$ in a neighborhood of the origin holds, then there exist global in time, bounded solutions.
\end{abstract}

\bigskip
\bigskip

\section{Introduction}
We shall deal with global existence and finite time blow-up of solutions to semilinear parabolic problems of the following form:
\begin{equation}\label{problema}
\begin{cases}
\, u_t= \Delta u +\,f(u) & \text{in}\,\, G\times (0,T) \\
\,\; u =u_0&\text{in}\,\, G\times \{0\}\,,
\end{cases}
\end{equation}
where $(G, \omega, \mu)$ is an {\it infinite weighted} graph with {\it edge-weight} $\omega$ and {\it node (or vertex) measure} $\mu$, $\Delta$  denotes the Laplace operator on $G$, 
$T\in (0,\infty]$, $f:[0,+\infty)\to [0,+\infty)$ is a locally Lipschitz, increasing function, $u_0:G\to [0, +\infty)$ is a given initial datum. 


In last years, the study of elliptic and parabolic equations on graphs, both finite and infinite, is receiving an increasing interest from various authors (see, e.g., \cite{Grig2, KLW, Mu2}). In particular, let us mention the contributions \cite{AS1, AS2, BMP, HJ, HK, MP2, MoPuSo} for elliptic equations, and \cite{BCG, CGZ, EM, F, GT, HMu, Huang, HuangKS, LSZ, LW2, LWu, MoPuSo2, Mu, Mu2, Wu} for parabolic equations.
On the other hand, there is a huge literature concerning the question of existence and nonexistence of global solutions to problem \eqref{problema} when the problem is posed in the Euclidean space or on Riemannian manifolds, mainly with $f(u)=u^p, p>1$. In $\mathbb R^N$ the celebrated so-called Fujita phenomenon has been discovered in \cite{F} (see also \cite{H} and \cite{KST} for the critical case); a complete account of results for semilinear parabolic equations posed in $\mathbb R^N$ can be found, e.g., in \cite{BB}, \cite{DL}, \cite{Levine} and in references therein.
On Riemannian manifolds, qualitatively different results hold; some of them can be found e.g. in \cite{BPT, MeGP1, MeGP2, MeGP3, MeGP4, GMPu, Sun1, MaMoPu, Punzo, Pu22, Zhang}.

Let us now recall some results regarding problem \eqref{problema} posed on graphs. In \cite{LSZ} the authors address the problem of blow up of nonnegative solutions to semilinear heat equation with convex source, posed in measure metric spaces, including in particular the case of weighted infinite graphs, i.e. problem \eqref{problema}. More precisely, in \cite[Theorem 4.1]{LSZ} nonexistence of global mild solutions to problem \eqref{problema} is proved under the following assumption:
$$
\frac{1}{\mu(D)}\sum_{x\in D} \left(\sum_{y\in D} p(x,y,T)u_0(y)\right)\mu(x)\,\ge\, F^{-1}(T),
$$
for some subset $D\subset G$ and $T>0$, where
$$
F(t):=\int_t^{+\infty}\frac1{f(s)}\,ds\,,
$$
$p$ is the heat kernel associated to $\Delta$ and $\mu$ is the \it node measure\rm, see Section \ref{prel} below for a full explanation.

Moreover, as a consequence, in \cite[Theorem 5.7]{LSZ} the authors show nonexistence of global mild solutions assuming that
$$
\begin{aligned}
&\inf_{x\in G}\mu(x)>0,\\
&F\left(\frac1t\right)\le Kt^\gamma\quad\text{for some}\,\,K>0,\,\gamma>0,\\
&\mu(B_R(x_0)\le CR^\theta\quad \text{for all R>1}, \text{for some}\,\,x_0\in G,\,C>0,\,\theta>0,
\end{aligned}
$$
where
$B_R(x_0):=\{x\in G\,:\, d(x, x_0)<R\}$ and $d$ is a distance on $G$. A similar result was also obtained in \cite[Theorem 1.1]{LWu}.

By completely different methods, in \cite{MoPuSo} it is shown  that \eqref{problema} with $f(u)=u^p$ admits no nontrivial nonnegative solutions, assuming an upper bound on the Laplacian of the distance, and an appropriate space-time volume growth condition. Furthermore, in \cite[Theorem 1.2]{LWu} global existence is addressed, but only when the problem is posed on a subgraph $\Omega\subset G$ with $\Omega^c\not=\emptyset$.

\smallskip

In this paper our goal is to extend to the case of an infinite graph the results obtained in \cite{MeGP4} on Riemannian manifolds. To this end, we shall in some case assume that $G$ is stochastically complete, i.e. that the heat semigroup preserves probability, namely:
\[
\sum_{y\in D} p(x,y,t)=1\ \text{for all }(x,t)\in G\times(0,+\infty),
\] and that
$$\lambda_1(G):= \inf \operatorname{spec}(-\Delta)>0\,.$$
We now describe our results making for the moment, for the sake of simplicity, extra hypotheses on $f$ (for the precise formulation, see Theorems \ref{teo1}, \ref{teo2}, \ref{teo3} below). If $f:[0, +\infty)\to [0, +\infty)$ is increasing and convex in $[0,+\infty)$, $f(0)=0$,
\begin{equation*}\label{N5}
\int^{+\infty}_1\frac{1}{f(s)}\,ds\,<+\infty,
\end{equation*}
and $$f'(0)>\lambda_1(G),$$
then any solution to problem \eqref{problema} blows up in finite time, provided $u_0\not\equiv 0$. On the other hand, for any given $\delta>0$, let $L\equiv L(f, \delta)$ denote the Lipschitz constant of $f$ over the interval $[0, \delta]$. If $f:[0, +\infty)\to [0, +\infty)$ is increasing in $[0,+\infty), f(0)=0$,
$$L\leq \lambda_1(G)$$ for a suitable $\delta>0$, and $u_0$ is small enough in $\ell^\infty(G)$, then
\eqref{problema} admits a global in time solution belonging to $L^{\infty}((0,\infty);\ell^{\infty}(G))$.
Obviously, if $f\in C^1([0, +\infty)),$ then the condition $L\leq \lambda_1(G)$ is automatically satisfied whenever $$f'(0)<\lambda_1(G).$$

Note that the proof of the blow-up follows in part the same line of arguments as in \cite{MeGP4}; however, in \cite{MeGP4} it is used in a crucial way that \eqref{eq26}, a crucial property of the long-time behaviour of the heat kernel, holds locally uniformly. This is not known to hold on graphs; however in the present proof we show how the validity of such property can be circumvented. On the contrary, the proof of global existence is entirely different from the one given in \cite{MeGP4}. In fact, in \cite{MeGP4} a barrier argument is used, combined with a priori estimates and compactness arguments that are not known to hold on graphs. Here, we make use of a fixed point argument to obtain the existence of a mild solution, when $f'(0)<\lambda_1(G)$. This does not work directly also when $f'(0)=\lambda_1(G)$. In fact, in that case, we construct approximate solutions on balls of radius $R$, again by means of a fixed point theorem. Then we can pass to the limit as the radius goes to $+\infty$, using an a priori bound, which is in turn obtained by virtue of the existence of a suitable supersolution, to be  constructed.

The paper is organized as follows. In section \ref{prel} we concisely summarize some mathematical background concerning graphs; in Section \ref{statements} we state our main results. Section \ref{blowup} is devoted to the proof of the blow-up result. In Section \ref{existence} we prove the global existence theorem when $f'(0)<\lambda_1(G)$; finally, the critical case $f'(0)=\lambda_1(G)$ is dealt with in Section \ref{proofcritical}.

\section{Mathematical background}\label{prel}

Let $G$ be a countably infinite set and , and
let $\mu:G\to (0,+\infty)$ be a measure on $G$
satisfying $\mu(\{x\}) <+\infty$ for every $x\in G$ (so that $(G,\mu)$ becomes a measure space). Furthermore, let
\begin{equation*}
\omega:G\times G\to [0,+\infty)
\end{equation*}
be a symmetric, with zero diagonal and finite sum function, i.e.
\begin{equation}\label{omega}
\begin{aligned}
&\text{(i)}\,\, \omega(x,y)=\omega(y,x)\quad &\text{for all}\,\,\, (x,y)\in G\times G;\\
&\text{(ii)}\,\, \omega(x,x)=0 \quad\quad\quad\,\, &\text{for all}\,\,\, x\in G;\\
&\text{(iii)}\,\, \displaystyle \sum_{y\in G} \omega(x,y)<\infty \quad &\text{for all}\,\,\, x\in G\,.
\end{aligned}
\end{equation}
Thus, we define  \textit{weighted graph} the triplet $(G,\omega,\mu)$, where $\omega$ and $\mu$ are the so called \textit{edge weight} and \textit{node measure}, respectively. Observe that assumption $(ii)$ corresponds to require that $G$ has no loops.
\smallskip

\noindent Let $x,y$ be two points in $G$; we say that
\begin{itemize}
\item $x$ is {\it connected} to $y$, and we write $x\sim y$, whenever $\omega(x,y)>0$;
\item the couple $(x,y)$ is an {\it edge} of the graph and the vertices $x,y$ are called the {\it endpoints} of the edge whenever $x\sim y$;
\item a collection of vertices $ \{x_k\}_{k=0}^n\subset G$ is a {\it path} if $x_k\sim x_{k+1}$ for all $k=0, \ldots, n-1.$
\end{itemize}

\noindent We say that the weighted graph $(G,\omega,\mu)$ is
\begin{itemize}
\item[(i)] {\em locally finite} if each vertex $x\in G$ has only finitely many $y\in G$ such that $x\sim y$;
\item[(ii)] {\em connected} if, for any two distinct vertices $x,y\in G$ there exists a path joining $x$ to $y$;
\item[(iii)] {\em undirected} if its edges do not have an orientation.
\end{itemize}
We shall always assume in the rest of the paper that the previous properties are fulfilled.


\medskip

Let $\mathfrak F$ denote the set of all functions $f: G\to \mathbb R$. For any $f\in \mathfrak F$ and for all $x,y\in G$, let us give the following
\begin{definition}\label{def1}
Let $(G, \omega,\mu)$ be a weighted graph. For any $f\in \mathfrak F$:
\begin{itemize}
\item the {\em difference operator} is
\begin{equation}\label{e2f}
\nabla_{xy} f:= f(y)-f(x)\,;
\end{equation}
\item the {\em (weighted) Laplace operator} on $(G, \omega, \mu)$ is
\begin{equation*}
\Delta f(x):=\frac{1}{\mu(x)}\sum_{y\in G}[f(y)-f(x)]\omega(x,y)\quad \text{ for all }\, x\in G\,.
\end{equation*}
\end{itemize}
\end{definition}
\noindent Clearly,
\[\Delta f(x)=\frac 1{\mu(x)}\sum_{y\in G}(\nabla_{xy} f)\omega(x,y)\quad \text{ for all } x\in G\,.\]

We set
$$\ell^\infty(G):=\{u\in \mathfrak F\,:\, \sup_{x\in G}|u(x)| <+\infty\}.
$$

\subsection{The heat semigroup on $G$}\label{heats}
We need to recall some preliminaries concerning the heat semigroup on $G$. Let $(G,\omega,\mu)$ be a weighted infinite graph.
Let $\{e^{t\Delta}\}_{t\ge 0}$ be the heat semigroup of $G$. It admits a (minimal) \textit{heat kernel}, namely a function $p:G\times G\times (0,+\infty)\to\R$, $p>0$ in $G\times G\times (0,+\infty)$ such that
\begin{equation*}
(e^{t\Delta} u_0)(x)=\sum_{y\in G}p(x,y,t)\,u_0(y)\,\mu(y), \quad x\in G,\,\, t>0,
\end{equation*}
for any $u_0\in \ell_\infty$. It is well known that
\begin{equation}\label{eq23}
\sum_{y\in G} p(x,y,t)\,\mu(y)\,\le \,1,\quad \text{for all}\,\,\, x\in G,\,\, t>0.
\end{equation}
We say that a graph $G$ is {\it stochastically} complete if the following condition holds:
\begin{equation*}
\sum_{y\in G} p(x,y,t) \,\mu(y)\,= \,1,\quad \text{for all}\,\,\, x\in G,\,\, t>0.
\end{equation*}
The main properties of $p(x,y,t)$ are stated below.
\begin{proposition}\label{prop1}
Let $(G,\omega,\mu)$ be a weighted infinite graph. Then the corresponding heat kernel satisfies the following properties:
\begin{itemize}
\item symmetry: $p(x,y,t)\equiv p(y,x,t)$ for all $x,y\in G$ and $t>0$.
\item $p(x,y,t)\ge 0$ for all $x,y\in G$ and $t>0$ and
\begin{equation}\label{eq23}
\sum_{y\in G} p(x,y,t) \,\mu(y)\,\le \,1,\quad \text{for all}\,\,\, x\in G,\,\, t>0.
\end{equation}
\item semigroup identity: for all $x,y\in G$ and $t,s>0$,
\begin{equation}\label{eq24}
p(x,y,t+s)=\sum_{z\in G} p(x,z,t)\, p(z,y,s) \,\mu(z).
\end{equation}
\end{itemize}
\end{proposition}

The next result is contained in \cite[Proposition 4.5]{KLW1}.
\begin{proposition}\label{prop2}
Let $(G,\omega,\mu)$ be a weighted infinite graph. For all vertices $x$ and $y$
\begin{equation}\label{eq26}
\lim_{t\to+\infty} \frac{\log p(x,y,t)}{t}\,=\, -\lambda_1(G)
\end{equation}
where $\lambda_1(G)$ is the infimum of the spectrum of the operator $-\Delta$.
\end{proposition}

Moreover, by combining together \cite[Theorems 2.1, 2.2]{F}, it is also possible to state the following result:
\begin{proposition}\label{prop2bis}
Let $(G,\omega,\mu)$ be a weighted infinite graph. Let $x,y\in G$, for any $\underline t>0$ there exists $\underline C>0$ such that
\begin{equation}\label{e301}
p(x,y,t)\leq \underline C\, e^{-\lambda_1(G) t} \quad \text{for all}\,\,\, t\geq \underline t\,,
\end{equation}
where $\lambda_1(G)$ is the infimum of the spectrum of the operator $-\Delta$.
\end{proposition}

\section{Statements of main results}\label{statements}
Solutions to problem \eqref{problema} are meant in the {\it mild} sense, according to the next definition. Here and hereafter $p:G\times G\times (0, +\infty)\to (0, +\infty)$ stands again for the heat kernel on $G$ discussed in the previous section.
\begin{definition}\label{defsol}
A function $u:G\times(0,\tau)\to\R$, $u\in L^\infty((0,\tau),\ell^\infty(G))$ is a \textit{mild} solution of problem \eqref{problema} if
\begin{equation}\label{solmild}
u(x,t)=\sum_{y\in G}p(x,y,t)u_0(y)+\int_0^t\sum_{y\in G} p(x,y,t-s)f(u(y,s))\,ds\,,
\end{equation}
for every $x\in G$ and $t\in(0,\tau)$. If $u$ is a mild solution of \eqref{problema} in $G\times(0,+\infty)$ then we call it a global solution.
\end{definition}

We say that a solution {\em blows up in finite time}, whenever there exists $\tau>0$ such that
$$\lim_{t\to \tau^-}\|u(t)\|_{\ell^\infty(G)}=+\infty\,.$$

\medskip

Our first result is concerned with nonexistence of global solutions.

\begin{theorem}\label{teo1}
Let $(G,\omega,\mu)$ be a weighted, stochastically complete, infinite graph with $\lambda_1(G)>0$. Let $u_0\in\ell^\infty(G)$, $u_0\ge 0$, $u_0\not\equiv 0$ in $G$. Let $f$ be locally Lipschitz in $[0, +\infty)$. Assume that $f\ge h$ where $h$ is increasing and convex in $[0,+\infty)$ and $h(0)=0$. Moreover, suppose that
\begin{equation}\label{N5}
\int^{+\infty}_1\frac{1}{h(s)}\,ds\,<+\infty,
\end{equation}
and finally that $h'(0)>\lambda_1(G)$.
Then any solution to problem \eqref{problema} blows up in finite time.
\end{theorem}
Note that in the above theorem the fact that $h$ is assumed to be increasing and convex implies the existence of $h'(0)$.


Concerning existence of global in time sulutions, we shall prove the following result.
\begin{theorem}\label{teo2}
Let $(G,\omega,\mu)$ be a weighted, infinite graph with $\lambda_1(G)>0$. Let $f:[0, +\infty)\to [0, +\infty)$ be a locally Lipschitz function. Assume that $f$ is increasing and $f(0)=0$. For a given $\delta>0$, let $L\equiv L(f, \delta)$ denotes the Lipschitz constant of $f$ over the interval $[0, \delta]$. 
Moreover, suppose that, for some $\delta>0$,
\begin{equation}\label{eq58}
L<{\lambda_1}(G).
\end{equation}
Furthermore, assume that $u_0\in \ell^\infty(G)$, $u_0\ge 0$ in $G$, and
\begin{equation}\label{eq12a}
\|u_0\|_{\ell^{\infty}(G)}\,\le\,\delta.
\end{equation}
Then there exists a global solution to problem \eqref{problema}; in addition, $u\in L^{\infty}((0,\infty);\ell^{\infty}(G))$.
\end{theorem}

\begin{remark}
(i) Note that if $f(u)=u^p$ with $p>1$, then \eqref{eq58} is fulfilled, for $\delta>0$ small enough. Hence for such $f(u)$, we have global existence of solutions for appropriate data.

(ii) If $f\in C^1([0, +\infty))$ and $f'(0)<\lambda_1(G)$, then \eqref{eq58} clearly holds for $\delta>0$ sufficiently small.
\end{remark}

In the following theorem, we address the case $L=\lambda_1(G)$.
\begin{theorem}\label{teo3}
Let $(G,\omega,\mu)$ be a weighted, infinite graph with $\lambda_1(G)>0$. Let $f:[0, +\infty)\to [0, +\infty)$ be a locally Lipschitz function. Assume that $f$ is increasing and $f(0)=0$. For a given $\delta>0$, let $L\equiv L(f, \delta)$ denote the Lipschitz constant of $f$ over the interval $[0, \delta]$. 
Moreover, suppose that
\begin{equation}\label{eq58bis}
L={\lambda_1}(G).
\end{equation}
Furthermore, assume that $u_0\in \ell^\infty(G)\cap\ell^1(G)$, $u_0\ge 0$ in $G$, and
\begin{equation}\label{eq12b}
\|u_0\|_{\ell^\infty(G)}\,\le\,\delta e^{-L\underline t},
\end{equation}
\begin{equation}\label{eq12c}
\|u_0\|_{\ell^1(G)}\,\le\,\frac{\delta}{\underline C},
\end{equation}
for suitable positive constants $\underline C$ and $\underline t$ (which are the ones appearing in \eqref{e301} above).
Then there exists a global solution to problem \eqref{problema}; in addition, $u\in L^{\infty}((0,\infty);\ell^{\infty}(G))$.
\end{theorem}

\section{Finite time blow-up for any initial datum}\label{blowup}

We shall deal here with the blowup analysis. We start with some technical tools.

\subsection{Two key estimates}

Let us first prove a preliminary lemma.

\begin{lemma}\label{lemma1}
Let $(G,\omega,\mu)$ be a weighted infinite graph with $\lambda_1(G)>0$. Suppose that $u_0\in \mathcal F, u_0\geq 0$, $u_0(x_0)>0$ for some $x_0\in G.$
Let $\varepsilon\in(0,\lambda_1(G))$. 
Then there exists $t_0=t_0(x_0,\varepsilon)>0$ such that
\begin{equation}\label{eq31}
(e^{t\Delta}u_0)(x_0) \ge \, C_1 e^{-[\lambda_1(G)+\varepsilon]t}\,, \quad \text{for any}\,\,\, t>t_0\,,
\end{equation}
where $C_1:=u_0(x_0) \mu(x_0)\,.$
\end{lemma}

\begin{proof}
Let $x_0 \in G$. From \eqref{eq26} if follows that there exists $t_0>0$ such that
$$
p(x_0,x_0,t)\ge e^{-[\lambda_1(G)+\varepsilon]t} \quad \text{for every}\, t>t_0\,.
$$
Hence
$$
\begin{aligned}
(e^{t\Delta}u_0)(x) &=\sum_{y\in G} p(x_0,y,t)u_0(y)\mu(y)\\
&\ge p(x_0,x_0,t) u_0(x_0)\, \mu(x_0) \\
&\ge e^{-[\lambda_1(G)+\varepsilon]t}\, u_0(x_0) \mu(x_0).
\end{aligned}
$$
Consequently, we obtain \eqref{eq31} with $C_1:=u_0(x_0) \mu(x_0)>0\,.$
\end{proof}

Let $u$ be a solution of equation \eqref{problema}. Then, for any $x\in G$ and for any $T>0$, we define
\begin{equation}\label{eq32}
\Phi_x^T(t)\equiv \Phi_x(t):=\sum_{z\in G} p(x,z,T-t)\, u(z,t)\, \mu(z)\,\quad \text{ for any }\,\, t\in [0, T]\,.
\end{equation}
Observe that
\begin{equation}\label{eq33}
\Phi_x(0)=\sum_{z\in G} p(x,z,T) \,u_0(z)\,\mu(z)\,= (e^{T\Delta}u_0)(x), \,\, x\in G\,.
\end{equation}
Suppose that $u_0\in \ell^\infty(G)$. Choose any
\begin{equation}\label{eq34a}
\delta>\|u_0\|_{\ell^{\infty}(G)}.
\end{equation}
From \eqref{eq34a} and \eqref{eq23} we obtain that, for any $x\in G$,
\begin{equation}\label{eq34b}
\begin{aligned}
\Phi_x(0)&=\sum_{z\in G} p(x,z,T) \,u_0(z)\, \mu(z)\\
&\le\,\|u_0\|_{\ell^{\infty}(G)}\sum_{z\in G} p(x,z,T)\, \mu(z)\\
&< \delta.
\end{aligned}
\end{equation}

We now state the following lemma.

\begin{lemma}\label{lemma2}
Let $(G,\omega,\mu), f, h, u_0$ be as in Theorem \ref{teo1}. Let $x\in G$ and $\Phi_x(t)$ be as in \eqref{eq32}. Set $\alpha:=h'(0).$ Then
\begin{equation}\label{eq34}
\Phi_x(0)\,\,\le\,\, C\,e^{-\alpha T}, \quad \text{for any}\,\,T\ge \bar t,
\end{equation}
for suitable $\bar t>0$ and $C>0$, depending on $x$.
\end{lemma}
Note that $\bar t$ and $C$ are given by \eqref{eq312a} and \eqref{eqf2} below, respectively, with $\delta$ as in \eqref{eq34a}.

\begin{proof}
Let $u$ be a solution to problem \eqref{problema}, hence
\begin{equation}\label{eq30f}
u(z,t)=\sum_{y\in G} p(z,y,t) u_0(y)\,\mu(y) + \int_0^t\sum_{y\in G} p(z,y,t-s)f(u)\,\mu(y)\, ds.
\end{equation}
In the definition of $\Phi_x^T(t)\equiv \Phi_x(t)$ (see \eqref{eq32}) fix any
\begin{equation}\label{eq20f}
T>\frac1{\alpha}\left[\log\delta -\log\Phi_x(0)\right]\,.
\end{equation}
We multiply \eqref{eq30f} by $p(x,z,T-t)$ and sum over $z\in G$. Therefore, we get
\begin{equation}\label{eq35}
\begin{aligned}
\sum_{z\in G} p(x,z,T-t)\,u(z,t)\,\mu(z) &=\sum_{z\in G}\sum_{y\in G} p(z,y,t)\,u_0(y)\,p(x,z,T-t)\,\mu(z)\mu(y)\\
&+ \int_0^t\sum_{z\in G}\sum_{y\in G} p(z,y,t-s)\,f(u)\,p(x,z,T-t)\,\mu(z)\mu(y)\,ds.
\end{aligned}
\end{equation}
Now, due to \eqref{eq32}, for all $t\in (0, T),$ equality \eqref{eq35} reads
$$
\begin{aligned}
\Phi_x(t)&=\sum_{z\in G}\sum_{y\in G} p(z,y,t)\,u_0(y)\,p(x,z,T-t)\,\mu(z)\mu(y) \\
&+ \int_0^t\sum_{z\in G}\sum_{y\in G} p(z,y,t-s)\,f(u)\,p(x,z,T-t)\,\mu(z)\mu(y)\,ds.
\end{aligned}
$$
By \eqref{eq33}, for all $t\in (0, T),$
\begin{equation*}
\begin{aligned}
\Phi_x(t)&= \sum_{y\in G} p(x,y,T)\,u_0(y)\mu(y)+\int_0^t\sum_{y\in G} f(u)\,p(x,y,T-s)\mu(y)\,ds \\
&= \Phi_x(0)+\int_0^t\sum_{y\in G} f(u)\,p(x,y,T-s)\mu(y)\,ds\,.
\end{aligned}
\end{equation*}
Since $f\geq h$ in $[0, +\infty$),
\begin{equation}\label{eq35b}
\begin{aligned}
\Phi_x'(t)&= \sum_{y\in G} f(u)\,p(x,y,T-t)\mu(y) \\
&\ge\sum_{y\in G} h(u)\,p(x,y,T-t)\mu(y)\,.
\end{aligned}
\end{equation}

Since $h$ is an increasing convex function, by using Jensen inequality, we get
\begin{equation}\label{eq37}
\sum_{y\in G} p(x,y,T-t)\,h\left(u(y,t)\right)\mu(y)\geq h\left(\sum_{y\in G} p(x,y,T-t)\,u(y,t)\mu(y)\right)=h(\Phi_x(t))\,.
\end{equation}

Combining together \eqref{eq35b} and \eqref{eq37}, we obtain
\begin{equation}\label{eq311}
\Phi_x'(t)\,\ge \, h(\Phi_x(t))\quad \text{for all }\,\, t\in (0, T)\,.
\end{equation}

Fix any $x\in G$. We first observe that \eqref{eq311} implies that $\Phi_x(t)$ is an increasing function w.r.t the time variable $t$, since
\begin{equation}\label{eq311c}
\Phi_x'(t)\,>\,0\quad \text{for any}\,\,t\in (0, T).
\end{equation}
Moreover, due to  \eqref{eq34a} and  \eqref{eq34b}, by continuity of $t\mapsto \Phi_x(t)$, we can infer that there exists $t_1>0$ such that
\begin{equation}\label{eqf1}
\Phi_x(t)<\delta \quad \text{for all }\,\, t\in (0, t_1)\,.
\end{equation}

Since $h$ is convex, increasing in $[0, +\infty), h(0)=0, h'(0)=\alpha$, then
\begin{equation}\label{N3}
h(s)\geq \alpha s\quad \text{ for all }\,\, s\geq 0\,.
\end{equation}
Due to \eqref{N3} and to \eqref{eqf1}, we get
\begin{equation*}
\begin{cases}
&\Phi_x'(t)\,\ge\, \alpha\, \Phi_x(t) \quad \text{for any}\,\,t\in(0,t_1), \\
&\Phi_x(0)\,<\,\delta.
\end{cases}
\end{equation*}
Let
\[\bar t:=\sup\{t>0\,:\, \Phi_x(t)<\delta\}\,.\]
We claim that
\begin{equation}\label{eq320}
0<\bar t\,\le \, -\frac{1}{\alpha}\log(\Phi_x(0)) +\frac{\log \delta}{\alpha}\,.
\end{equation}

In order to show \eqref{eq320}, consider the Cauchy problem
\begin{equation*}
\begin{cases}
y'(t) = \alpha y(t),& t>0\\
y(0)=\Phi_x(0)\,.
\end{cases}
\end{equation*}
Clearly,
\[y(t)=\Phi_x(0) e^{\alpha t}, \quad t>0\,.\]
Hence
\begin{equation*}
y(\tau)=\delta \quad \text{whenever }\,\, \tau=\frac 1{\alpha}\left[\log(\delta)-\log(\Phi_x(0))\right]\,.
\end{equation*}
Furthermore, note that, in view of \eqref{eq34b} and \eqref{eq20f},
\begin{equation}\label{eq31f}
0<\tau <T\,.
\end{equation}
By comparison,
\[\Phi_x(t)\geq y(t)\quad \text{for all }\, t\in (0, T)\,.\]
Thus, we can infer that there exists $\bar t\in (0, \tau\,]$ such that
\begin{equation}\label{eq312a}
\Phi_x(\bar t)=\delta.
\end{equation}
In particular, from \eqref{eq31f} it follows that $\bar t<T\,.$
\smallskip

Due to \eqref{eq311c} and \eqref{eq312a}, we obtain that
\begin{equation}\label{eq312}
\Phi_x(t)\,>\,\delta\quad \text{for any}\,\,\, \bar t<t <T.
\end{equation}
By \eqref{eq311}, in particular we have
\begin{equation}\label{eq314}
\Phi_x'(t)\,\ge \,h(\Phi_x(t))\quad \text{for any}\,\,\,\bar{t}<t<T.
\end{equation}

Define
\begin{equation*}
H(t):=\int_{\Phi_x(t)}^{+\infty} \frac{1}{h(z)}\,dz\,\quad \text{for all }\,\, \bar t<t <T \,.
\end{equation*}
Note that $H$ is well-defined thanks to hypotheses \eqref{N5} and to \eqref{eq312}. Furthermore,
\begin{equation}\label{eq39}
H'(t)=-\frac{\Phi_x'(t)}{h(\Phi_x(t))}\,\quad \text{for any }\,\, \bar t<t<T\,.
\end{equation}
We now define
\begin{equation}\label{eq315}
w(t):=\exp\{H(t)\}\quad \text{for any}\,\,\,\bar{t}<t<T.
\end{equation}
Then, due to \eqref{eq39} and \eqref{eq314},
\begin{equation}\label{eq39b}
\begin{aligned}
w'(t)&=-\frac{\Phi_x'(t)}{h(\Phi_x(t))}w(t)\\
&\le - \,\,w(t)\,\quad\quad\quad \text{for any}\,\,\,\bar{t}<t<T.
\end{aligned}
\end{equation}
By integrating \eqref{eq39b} we get:
\begin{equation}\label{eq316}
\begin{aligned}
w(t)&\le w(\bar t\,)\exp\left\{-\int_{\bar t}^t\,ds\right\}\\
&\le w(\bar t\,)\exp\left\{-\,(t-\bar t)\right\} \quad \text{for any}\,\,\,\bar{t}<t<T.
\end{aligned}
\end{equation}
We substitute \eqref{eq315} into \eqref{eq316}, so we have
\begin{equation*}
\exp\{H(t)\}\,\le\,\exp\left\{H(\bar t)-\,(t-\bar t)\right\}\quad \text{for any}\,\,\,\bar{t}<t<T.
\end{equation*}
Thus, for any $\bar t<t<T$,
\begin{equation}\label{eq319}
H(t)\,\le\,H(\bar t)-\,(t-\bar t\,).
\end{equation}

We now combine \eqref{eq319} together with \eqref{eq320}, whence
$$
0\le H(t)\,\le\,H(\bar t)-\,\left(t +\frac{1}{\alpha}\log(\Phi_x(0)) -\frac{\log\delta}{\alpha}\right) \quad  \text{for any} \,\,\, \bar t<t<T.
$$
Hence
\begin{equation}\label{eq325}
\log(\Phi_x(0))\,\le \, \alpha H(\bar t)+\log\delta -\alpha \,\,t\, \quad  \text{for any} \,\,\, \bar t<t<T.
\end{equation}
We now take the exponential of both sides of \eqref{eq325}. Thus we get, for any $\bar t<t<T$,
$$
\begin{aligned}
\Phi_x(0)&\le \exp\left\{\alpha H(\bar t)+\log\delta -\alpha \,\,t\right\}\\
&= C\,\exp\{-\alpha\,t\}\,,
\end{aligned}
$$
where
\begin{equation}\label{eqf2}
C:=\exp\{\alpha H(\bar t)+\log\delta\}\,.
\end{equation}
This is the inequality \eqref{eq34}.
\end{proof}

\subsection{Proof of Theorem \ref{teo1}}
\begin{proof}[Proof of Theorem \ref{teo1}]
Take any $\delta>0$ fulfilling \eqref{eq34a}.  We suppose, by contradiction, that $u$ is a global solution of problem \eqref{problema}.
Since $\alpha:=h'(0)>\lambda_1(G)$, there exists $\varepsilon\in (0,\alpha-\lambda_1(G))$ such that
$$
\alpha>\lambda_1(G)+\varepsilon.
$$
Thus obviously,
\begin{equation}\label{eq327}
\lim_{T\to +\infty} e^{\alpha-(\lambda_1(G)+\varepsilon)]T}=+\infty.
\end{equation}
Let $x_0\in G$. By Lemma \ref{lemma1} and Lemma \ref{lemma2},
$$
C_1 \, e^{-[\lambda_1(G)+\varepsilon]\,T}\,\,\le\,\,(e^{T\Delta}u_0)(x)= \Phi_x(0)\,\,\le\,\, C\,e^{-\alpha\, T}, \quad \text{for any}\,\,\, T>\max\{t_0, \bar t\,\},
$$
where $t_0>0$, $C_1>0$ are given in Lemma \ref{lemma1}, while $\bar t>0$, $C>0$ in Lemma \ref{lemma2}. Hence, if $u$ exists globally in time, we would have
\begin{equation}\label{eq326}
e^{[\alpha-(\lambda_1(G)+\varepsilon)]\,T}\,\,\le\,\, \frac{C}{C_1} \quad \text{for any}\,\,\, T>\max\{t_0, \bar t\,\}.
\end{equation}
Nonetheless, due to \eqref{eq327}, the left hand side of \eqref{eq326} tends to $+\infty$ as $T\to \infty$.
Thus, we have a contradiction. Hence the thesis follows.

\end{proof}

\section{Global existence}\label{existence}
First of all let us introduce the Banach space where the necessary fixed point theorem will be applied.
\begin{definition}\label{funcspace}
Let $M>0, \gamma>0$. Let  $y_0\in G$ be a point such that $u_0(y_0)>0$. We denote by $\mathcal{M}(M, \gamma, y_0)\equiv \mathcal M$ the set of all non-negative functions $u: G\times[0,+\infty)\to \mathbb R$ such that $t\mapsto u(x, t)$ is continuous for each $x\in G$, and
\begin{equation}\label{eq51}
0\le u(x,t)\le M p(x,y_0,t+\gamma)e^{\lambda_1(G) t} \quad \text{ for all } x\in G, t\geq 0 .
\end{equation}
\end{definition}
\noindent We also assume that
\begin{equation}\label{eq51bis}
0\le u_0(x)\le \varepsilon\, p(x,y_0,t)\,,
\end{equation}
where
\begin{equation}\label{eps}
0<\varepsilon\le \left(1-\frac{L}{\lambda_1(G)}\right)M,
\end{equation}
where \eqref{eq58} is assumed, so that the r.h.s. of the above inequality is strictly positive. It is easily seen that $\mathcal M$ endowed with the norm
\begin{equation}\label{eq52}
\|u\|_{\mathcal{M}}:=\sup_{t>0,\,x\in G}\frac{|u(x,t)|}{p(x,y_0,t+\gamma)e^{\lambda_1(G) t}}
\end{equation}
is a Banach space. Let
\begin{equation}\label{eq53}
(\Psi u)(x,t):=\sum_{z\in G}p(x,z,t)u_0(z)\mu(z)+\int_0^t\sum_{z\in G}p(x,z,t-s)f(u(z,s))\mu(z)\,ds\,.
\end{equation}
We first prove two key estimates that ensure that the map $\Psi:\mathcal{M}\to\mathcal{M}$ is a contraction map in the space $(\mathcal{M}, \|\cdot\|_{\mathcal M})$.

\begin{lemma}\label{lemma51}
Let assumptions of Theorem \ref{teo2} be satisfied. Moreover assume that $u_0\in\ell^\infty(G)$ is such that \eqref{eq51bis} and \eqref{eps} hold. Let $\mathcal M$ be as in Definition \ref{funcspace} with
\begin{equation}\label{eq302}
M\leq \frac{\delta}{\underline C},
\end{equation}
$\underline C$ being given by \eqref{e301}.
Then, for any  $u\in\mathcal{M}$,
\begin{equation}\label{eq54}
\Psi(u)\in\mathcal{M}.
\end{equation}
\end{lemma}

\begin{proof}
Observe that, since $u\in\mathcal{M}$ and \eqref{eq302} holds, due to \eqref{e301}, we get
\begin{equation}\label{eq56}
\begin{aligned}
0\leq u(x,t) &\le M\,p(x,y_0,t+\gamma)e^{\lambda_1(G) t}\\
&\le M\,\underline C\,e^{-\lambda_1(G)\gamma}\,e^{-(\lambda_1(G)-\lambda_1(G)) t}\\
&\le \delta\quad \text{ for any } x\in G, t>0.
\end{aligned}
\end{equation}
Therefore, in particular,
$$\|u\|_{\mathcal M}\le M.$$
By exploiting the definition of $\Psi$ in \eqref{eq53}, thanks to 
\eqref{eq51bis}, 
we have that
\begin{equation}\label{eq59}
\begin{aligned}
0\le(\Psi u)(x,t)&=\sum_{z\in G}p(x,z,t)u_0(z)\mu(z)+\int_0^t\sum_{z\in G} p(x,z,t-s) f(u(z,s))\mu(z)\,ds\\
&=\sum_{z\in G}p(x,z,t)u_0(z)\mu(z)\\ &+\int_0^t\sum_{z\in G} p(x,z,t-s) \frac{f(u(z,s))}{p(x,y_0,s+\gamma)e^{\lambda_1(G) s}}p(z,y_0,s+\gamma)e^{\lambda_1(G) s}\mu(z)\,ds
\end{aligned}
\end{equation}
for any $x\in G, t>0$. Observe that in view of \eqref{eq56} and the very definition of $L$,
\begin{equation}\label{eq303}
f(u(z, s))\leq L u(z,s) \quad \text{ for all } z\in G, s\geq 0\,.
\end{equation}
By \eqref{eq24}, \eqref{eps}, \eqref{eq59} and \eqref{eq303},
\begin{equation}\label{e59b}
\begin{aligned}
(\Psi u)(x,t)=
&\le\varepsilon\sum_{z\in G}p(x,z,t)p(z,y_0,\gamma)\mu(z)\\
&+\int_0^t\sum_{z\in G} p(x,z,t-s) \frac{L u(z,s)}{p(z,y_0,s+\gamma)e^{\lambda_1(G) t}}p(z,y_0,s+\gamma)e^{\lambda_1(G) s}\mu(z)\,ds\\
&\le \varepsilon p(x,y_0,t+\gamma)+L\,\|u\|_{\mathcal M}\int_0^te^{\lambda_1(G) s}\sum_{z\in G} p(x,z,t-s)p(z,y_0,s+\gamma)\mu(z)\,ds\\
&\le\varepsilon p(x,y_0,t+\gamma)e^{\lambda_1(G) t}+L \,p(x,y_0,t+\gamma)\|u\|_{\mathcal M}\int_0^te^{\lambda_1(G) s}\,ds\\
&\le\varepsilon p(x,y_0,t+\gamma)e^{\lambda_1(G) t}+ \frac{L}{\lambda_1(G)}\,p(x,y_0,t+\gamma) e^{\lambda_1(G) t}\|u\|_{\mathcal M}\\
&\le \left(\varepsilon + \frac{L}{\lambda_1(G)}M\right)p(x,y_0,t+\gamma) e^{\lambda_1(G) t}\\
&\le M\,p(x,y_0,t+\gamma)e^{\lambda_1(G) t} \quad \text{ for any } x\in G, t>0\,.
\end{aligned}
\end{equation}
Hence, $(\Psi u)\in\mathcal{M}$.
\end{proof}

\begin{proposition}\label{prop51}
Let assumptions of Theorem \ref{teo2} be satisfied. Moreover assume that $u_0\in\ell^\infty(G)$ is such that \eqref{eq51bis} and \eqref{eps} hold. Let $\mathcal M$ be as in Definition \ref{funcspace} with $M$ given by \eqref{eq302}.
Then, for any $u,v\in\mathcal{M}$,
\begin{equation}\label{eq55}
\|\Psi u-\Psi v\|_{\mathcal M}\le\frac{L}{\lambda_1}\|u-v\|_{\mathcal M}\,.
\end{equation}
\end{proposition}

\begin{proof} Let $u, v\in \mathcal M$.
Clearly, by Lemma \ref{lemma51}, $\Psi u, \Psi v\in \mathcal M$. Furthermore, as we have observed in the proof of Lemma \ref{lemma51}, \eqref{eq56} holds both for $u$ and $v$. This, combined with \eqref{eq52}, implies that
\begin{equation}\label{eq57}
\begin{aligned}
|f(u(x,t))-f(v(x,t))|&\le  L |u(x,t)-v(x,t)| \\ &\le L\,p(x,y_0,t+\gamma)e^{\lambda_1(G) t}\|u-v\|_{\mathcal M}\quad \text{ for any }\, x\in G, t>0\,.
\end{aligned}
\end{equation}
Now, due to \eqref{eq24}, \eqref{eq53} 
and \eqref{eq57}, we have that
\begin{equation}\label{eq510}
\begin{aligned}
0\le|(\Psi u)-(\Psi v)|(x,t)&=\left|\int_0^t\sum_{z\in G} p(x,z,t-s) [f(u)-f(v)](z,s)\mu(z)\,ds\right|\\
&\leq \int_0^t\sum_{z\in G} p(x,z,t-s) \big|f(u)-f(v)\big|(z,s)\mu(z)\,ds\\
&=\int_0^t\sum_{z\in G} p(x,z,t-s)L\,p(z,y_0,s+\gamma)e^{\lambda_1(G) s}\|u-v\|_{\mathcal M}\mu(z)\,ds\\
&\le L\,\|u-v\|_{\mathcal M}\int_0^te^{\lambda_1(G) s}\sum_{z\in G} p(x,z,t-s) p(z,y_0,s+\gamma)\mu(z)\,ds\\
&\le L\,p(x,y_0,t+\gamma)\|u-v\|_{\mathcal M} \int_0^te^{\lambda_1(G) s}\,ds\\
&\le \frac L{\lambda_1(G)} \,p(x,y_0,t+\gamma) e^{\lambda_1(G) t}\|u-v\|_{\mathcal M}.
\end{aligned}
\end{equation}
By dividing both sides of \eqref{eq510} by $p(x,y_0,t+\gamma) e^{\lambda_1(G) t}$, we get
$$
\begin{aligned}
0&\le\frac{|(\Psi u)-(\Psi v)|(x,t)}{p(x,y_0,t+\gamma) e^{\lambda_1(G) t}}\\
&\le\sup_{t>0,\,x\in G}\,\frac{|\Psi u-\Psi v|(x,t)}{p(x,y_0,t+\gamma) e^{\lambda_1(G) t}}\\
&=\|\Psi u-\Psi v\|_{\mathcal M}\\
&\le\frac{L}{\lambda_1(G)} \|u-v\|_{\mathcal M}.
\end{aligned}
$$
\end{proof}

\begin{proof}[Proof of Theorem \ref{teo2}]
In view of assumption \eqref{eq58} and Proposition \ref{prop51}, the map
\[\Psi:\mathcal M\to \mathcal M\]
defined in \eqref{eq53} is a contraction. Hence, by the Banach-Caccioppoli theorem, there exists a unique fixed point $u\in \mathcal M$ of $\Psi$. Such $u$ is a mild solution of problem \eqref{problema}. Furthermore, in view of \eqref{eq51}, $u$ is bounded in $G\times (0, +\infty).$
\end{proof}

\section{Proof of Theorem \ref{teo3}}\label{proofcritical}

\subsection{Existence of local in time solutions on balls}

Let $x_0\in G$ be a reference point, let $d$ be a distance on $G$.
For any $R>0$, consider the ball
$$B_R(x_0):=\{x\in G\,:\, d(x, x_0)< R\}.$$ Furthermore, for every $R>0$ let $\zeta_R\in C^{\infty}(G)$ be such that $\operatorname{supp} \zeta_R$ is a finite subset of $G$, $0\le \zeta_R\le 1$, $\zeta_R\equiv 1$ in $B_{R/2}$ For any $R>0$ there exists a unique mild solution $u_R$ to problem
\begin{equation}\label{eq612}
\begin{cases}
& \partial_t u=\Delta u + f(u)\quad \text{in}\,\,\,B_R\times (0,T)\\
& u=0\quad\quad\quad\quad \quad \quad \text{in}\,\,\,G\setminus B_R\times (0,T)\\
&u=u_0\zeta_R \quad\quad\quad \quad\,\,\, \text{in}\,\,\,B_R\times \{0\}\,.
\end{cases}
\end{equation}

Let $M>0,\gamma>0$ and let $y_0\in B_R$ be a point where $u_0$ is positive. We consider the linear space $\mathcal{V}_R\subset L^{\infty}((0,T),\ell^\infty(B_R))$ of all non-negative functions $u: B_R\times[0,T)\to \mathbb R$ such that $t\mapsto u(x, t)$ is continuous for each $x\in B_R$, and
\begin{equation}\label{eq615}
0\le u(x,t)\le M\,
\quad \text{ for all } x\in B_R, t\in[0,T) .
\end{equation}
\noindent We also assume that
\begin{equation}\label{eq616}
0\le u_0(x)\le \varepsilon\,,
\end{equation}
where, for $T<1/L$:
\begin{equation}\label{eps2}
0<\varepsilon\le \left(1-LT\right)M.
\end{equation}

On the space $\mathcal{V}_R$ we consider the supremum norm
\begin{equation*}
\|u\|_{\infty}:=\sup_{0<t<T,\,x\in B_R}|u(x,t)|
\end{equation*}
Let us define the map
\begin{equation}\label{eq622}
(\Psi u)(x,t):=\sum_{z\in B_R}p_R(x,z,t)u_0(z)\mu(z)+\int_0^t\sum_{z\in B_R}p_R(x,z,t-s)f(u(z,s))\mu(z)\,ds\,,
\end{equation}
where $p_R$ is the heat semigroup on the ball $B_R$ completed with homogeneous Dirichlet boundary conditions. We observe that
the following two estimates ensure that the map $\Psi:\mathcal{V}_R\to \mathcal{V}_R$ is a contraction map in the space $(\mathcal{V}_R, \|\cdot\|_{\infty})$.

\begin{lemma}\label{lemma62}
Let assumptions of Theorem \ref{teo3} be satisfied. Moreover assume that $u_0\in\ell^\infty(G)$ is such that \eqref{eq615} and \eqref{eps2} hold. Let
\begin{equation}\label{eq617}
M\leq \delta
\end{equation}
Then, for any  $u\in \mathcal{V}_R$
\begin{equation}\label{eq618}
\Psi(u)\in \mathcal{V}_R
\end{equation}
\end{lemma}

\begin{proof}
Observe that, since $u\in \mathcal{V}_R$ and \eqref{eq617} holds, we get
\begin{equation}\label{eq619}
\begin{aligned}
0\leq u(x,t) &\le M
\le \delta\quad \text{ for any } x\in B_R,\, t\in[0,T).
\end{aligned}
\end{equation}
Therefore, in particular,
$$\|u\|_{\infty}\le M.$$
Observe that in view of \eqref{eq619} and the very definition of $L$,
\begin{equation}\label{eq621}
f(u(z, s))\leq L u(z,s) \quad \text{ for all } z\in B_R, s\in[ 0,T)\,.
\end{equation}
By exploiting the definition of $\Psi$ in \eqref{eq622}, thanks to \eqref{eq616}, \eqref{eps2}, \eqref{eq621}, \eqref{eq24} we have that
\begin{equation}\label{eq620}
\begin{aligned}
0\le(\Psi u)(x,t)&=\sum_{z\in B_R}p_R(x,z,t)u_0(z)\mu(z)+\int_0^t\sum_{z\in B_R} p_R(x,z,t-s) f(u(z,s))\mu(z)\,ds\\
&\le\varepsilon\sum_{z\in B_R}p_R(x,z,t)\mu(z)+\int_0^t\sum_{z\in B_R} p_R(x,z,t-s) L u(z,s)\mu(z)\,ds\\
&\le \varepsilon+L\,\|u\|_{\infty}\int_0^t\sum_{z\in B_R} p_R(x,z,t-s)\mu(z)\,ds\\
&\le\varepsilon+L \,T\,M\,\le M\,\quad \text{ for any } x\in B_R, \,\,t\in(0,T)\,.
\end{aligned}
\end{equation}
Hence, $(\Psi u)\in \mathcal{V}_R$.
\end{proof}

\begin{proposition}\label{prop61}
Let assumptions of Theorem \ref{teo3} be satisfied. Moreover assume that $u_0\in\ell^\infty(G)$ is such that \eqref{eq615} and \eqref{eps2} hold. Let $M$ be given by \eqref{eq617}.
Then, for any $u,v\in \mathcal{V}_R$,
\begin{equation}\label{eq623}
\|\Psi u-\Psi v\|_{\infty}\le L\,T\|u-v\|_{\infty}\,.
\end{equation}
\end{proposition}

\begin{proof} Let $u, v\in \mathcal{V}_R$.
Clearly, by Lemma \ref{lemma62}, $\Psi u, \Psi v\in \mathcal{V}_R$. Furthermore, as we have observed in the proof of Lemma \ref{lemma62}, \eqref{eq619} holds both for $u$ and $v$. This implies that
\begin{equation}\label{eq624}
\begin{aligned}
|f(u(x,t))-f(v(x,t))|&\le  L |u(x,t)-v(x,t)| \\ &\le L\,\|u-v\|_{\infty}\quad \text{ for any }\, x\in B_R, \,\,t\in(0,T)\,.
\end{aligned}
\end{equation}
Now, due to \eqref{eq24}, \eqref{eq53} and \eqref{eq624}, we have that
\begin{equation}\label{eq510}
\begin{aligned}
0\le|(\Psi u)-(\Psi v)|(x,t)&=\left|\int_0^t\sum_{z\in B_R} p_R(x,z,t-s) [f(u)-f(v)](z,s)\mu(z)\,ds\right|\\
&\leq \int_0^t\sum_{z\in B_R} p_R(x,z,t-s) \big|f(u)-f(v)\big|(z,s)\mu(z)\,ds\\
&\le L\,\|u-v\|_{\infty}\,\int_0^t\sum_{z\in B_R} p_R(x,z,t-s)\,\mu(z)\,ds\\
&\le L\,T\|u-v\|_{\infty}.
\end{aligned}
\end{equation}
Then we get
$$
\begin{aligned}
0&\le|(\Psi u)-(\Psi v)|(x,t)\le\sup_{t>0,\,x\in B_R}\,|\Psi u-\Psi v|(x,t)\\
&=\|\Psi u-\Psi v\|_{\infty}\le L\,T \|u-v\|_{\infty}.
\end{aligned}
$$
\end{proof}

\subsection{A global supersolution}
\begin{lemma}\label{lem61}
Problem \eqref{problema} admits a global supersolution $\overline u$.
\end{lemma}
\begin{proof}
Consider the linear Cauchy problem for the heat equation
\begin{equation}\label{eq61}
\begin{cases}
&v_t=\Delta v\quad\,\, \text{in}\,\,\, G\times(0,+\infty)\\
&v=u_0\quad\quad \text{in}\,\,\,G\times\{0\},
\end{cases}
\end{equation}
with $u_0$ as in Theorem \ref{teo3}.
Observe that problem \eqref{eq61} admits the classical solution
\begin{equation}\label{eq62}
v(x,t)=\sum_{y\in G} \,p(x,y,t)\,u_0(y)\,\mu(y), \quad x\in G,\, t\geq 0.
\end{equation}
Hence, since $u_0\in \ell^\infty(G)$,
\begin{equation}\label{eq65}
\|v(t)\|_{\ell^{\infty}(G)}\le\|u_0\|_{\ell^{\infty}(G)}\quad \text{for any}\,\,\,t>0\,.
\end{equation}
Moreover, since $u_0\in \ell^1(G)$, due to \eqref{e301},
\begin{align}
v(x,t) \le \underline{C}\, \|u_0\|_{\ell^1(G)}\,e^{-\lambda_1 t}\quad \text{for any}\,\,\,x\in G, t>\underline t, \label{eq64}
\end{align}
where $\underline C$ and $\underline t$ have been defined in \eqref{e301}.
Define
\begin{equation*}
\bar{u}(x,t):= e^{L \,t}\,v(x,t),\quad x\in G, t\geq 0,
\end{equation*}
with $v$ as in \eqref{eq62}.
Note that, due to \eqref{eq12b} and \eqref{eq65}, for any $x\in G, t\in(0,\underline t]$,
\begin{equation}\label{eq69}
0\leq \bar u(x,t)\leq e^{L\,t}\|v(t)\|_{\ell^{\infty}(G)}\,\le\, e^{L\, t}\|u_0\|_{\ell^{\infty}(G)}\,\le\,\delta.
\end{equation}
Moreover, due to \eqref{eq12c}, \eqref{eq64}, since $L= \lambda_1(M)$, for any $t\ge\underline t$ we get
\begin{equation}\label{eq610}
0\leq \bar u(x,t)\leq e^{L\,t}\|v(t)\|_{\ell^{\infty}(G)}\le\, \underline{C}\,\|u_0\|_{\ell^1(G)}e^{-(\lambda_1-L)\, t}\le \,\delta.
\end{equation}
Inequalities \eqref{eq69} and \eqref{eq610} yield
\begin{equation}\label{eq68}
0\,\le\,\bar{u}(x,t)\,\le\, \delta\quad \text{for any}\,\,\,x\in G,\,\, t>0.
\end{equation}
Furthermore, we have
$$
\bar{u}_t-\Delta\bar{u}-f(\bar u)=L\,e^{L\,t}\,v+e^{L\,t}\,v_t-e^{L\,t}\Delta v-f(\bar u).
$$
Observe that, due to the assumptions on $f$ in Theorem \ref{teo3}, one has that, for any $s\in[0,\delta]$
$$
f(s)\le Ls\,.
$$
Therefore, due to \eqref{eq68} and by using the fact that $v$ is a solution to problem \eqref{eq61}, we get
\begin{equation}\label{eq611}
\bar{u}_t-\Delta\bar{u}-f(\bar u)= L \bar u - f(\bar u) \geq \,0.
\end{equation}
Hence $\bar u$ is a classical (and therefore mild) supersolution to problem \eqref{problema} in $G\times(0,\infty)$.
\end{proof}


\subsection{Global existence in the critical case $L=\lambda_1(G)$}
\begin{proof}[Proof of Theorem \ref{teo3}]
First observe that, due to Proposition \ref{prop61}, if as already required the inequality
$$
T<\frac1L
$$
holds, the mapping
\[\Psi:\mathcal{V}_R\to\mathcal{V}_R\]
defined in \eqref{eq622} is a contraction. Hence, by the Banach-Caccioppoli theorem, there exists a unique fixed point $u_R\in \mathcal{V}_R$ of $\Psi$. Such $u_R$ is a bounded mild solution of problem \eqref{eq612} in $B_R\times (0, T).$
Now, for any $R>0$, in view of \eqref{eq611}, since
\[ v= u_0\geq \zeta_R u_0 \quad \text{ in }\,\, G\times \{0\},\]
$\bar u$ is a bounded weak supersolution of problem \eqref{eq612}. Obviously, for any $R>0$, $\underline u\equiv 0$ is a subsolution to problem \eqref{eq612}.
Hence, by the comparison principle, for every $R>0$ we obtain
\begin{equation}\label{eq613}
0\,\le\,u_R\,\le\,\bar{u}\quad \text{for any}\,\,\,(x,t)\in B_R\times(0,T).
\end{equation}
This ensures that $u_R$ is indeed defined for every $(x,t)\in B_R\times(0,+\infty)$. Therefore
\begin{equation}\label{eq614}
u_R(x,t)=\sum_{y\in B_R}p_R(x,y,t)u_0(y)+\int_0^t\sum_{y\in B_R} p_R(x,y,t-s)f(u_R(y,s))\,ds,
\end{equation}
for every $x\in B_R$ and $t\in(0,+\infty).$

Then, let us observe that $\{p_R\}_{R>0}$ is a monotone increasing sequence which converge to $p$ as $R\to \infty$. Moreover, $\{u_R\}_{R>0}$ is monotone increasing hence it admits a limit, we name it $u$. Since $\{f(u_R)\}_{R>0}$ is monotone increasing, we infer that $f(u_R)\to f(u)$ as $R\to\infty$. By monotone convergence theorem we can take the limit as $R\to\infty$ in \eqref{eq614} and we get \eqref{solmild} thus the proof is complete.
\end{proof}
\par\bigskip\noindent
\textbf{Acknowledgments.}
The first and the third authors are members of the Gruppo Nazionale per l'Analisi Matematica, la Probabilit\`a e le loro Applicazioni (GNAMPA, Italy) of the Istituto Nazionale di Alta Matematica (INdAM, Italy). The second author is funded by the Deutsche Forschungsgemeinschaft (DFG, German Research Foundation) - SFB 1283/2 2021 - 317210226.
The first author that this work is part of the Prin project 2022 ``Partial Differential Equations and Related Geometric-Functional Inequalities'', ref. 20229M52AS, and the third author acknowledges that this work is part of the PRIN project 2022 ``Geometric-analytic methods for PDEs and applications'', ref. 2022SLTHCE, both financially supported by the EU, in the framework of the "Next Generation EU initiative".

{\bf Conflict of interest}. The author states no conflict of interest.

{\bf Data availability statement}. There are no data associated with this research.

%
%

%


\end{document}